\documentclass[11pt]{amsart}

\usepackage{amsmath}
\usepackage{amssymb}
\usepackage{latexsym}
\usepackage{amscd}
\usepackage{mathrsfs}
\usepackage[all]{xy}

\newdimen\AAdi%
\newbox\AAbo%
\def\AAk#1#2{\s_etbox\AAbo=\hbox{#2}\AAdi=\wd\AAbo\kern#1\AAdi{}}%
\def\AAr#1#2#3{\s_etbox\AAbo=\hbox{#2}\AAdi=\ht\AAbo\raise#1\AAdi\hbox{#3}}%
\font\tenmsb=msbm10 at 12pt \font\sevenmsb=msbm7 at 8pt
\font\fivemsb=msbm5 at 6pt
\newfam\msbfam
\textfont\msbfam=\tenmsb \scriptfont\msbfam=\sevenmsb
\scriptscriptfont\msbfam=\fivemsb

\newtheorem{theorem}{Theorem}

\newtheorem{remark}[theorem]{Remark}

\newtheorem{lemma}[theorem]{Lemma}

\numberwithin{equation}{section} \numberwithin{theorem}{section}

\renewcommand{\topmargin}{0cm}
\renewcommand{\oddsidemargin}{5mm}
\renewcommand{\evensidemargin}{5mm}
\renewcommand{\textwidth}{150mm}
\renewcommand{\textheight}{230mm}

\def\R{\mathbb R}

\def\f#1#2{\frac{#1}{#2}}

\def\a{\alpha}

\def\r{\Re_{I\!V}}

\def\p#1{\partial #1}

\def\de{\delta}
\def\De{\Delta}

\def\ep{\epsilon}

\def\g{\gamma}
\def\k{\kappa}
\def\la{\lambda}
\def\La{\Lambda}
\def\lan{\langle}
\def\ran{\rangle}

\def\Om{\Omega}
\def\th{\theta}
\def\Th{\Theta}

\def\Si{\Sigma}

\def\r{\rho}
\def\z{\zeta}

\usepackage{color}

\begin{document}

\title
[Boundary regularity for mean curvature flows]
{Boundary regularity for mean curvature flows of higher codimension}

\author{Qi Ding}
\address{Shanghai Center for Mathematical Sciences, Fudan University, Shanghai 200438, China}
\email{dingqi@fudan.edu.cn}
\author{J. Jost}
\address{Max Planck Institute for Mathematics in the Sciences, Inselstr. 22, 04103 Leipzig, Germany}
\email{jost@mis.mpg.de}
\author{Y.L. Xin}
\address{Institute of Mathematics, Fudan University, Shanghai 200433, China}
\email{ylxin@fudan.edu.cn}

\thanks{ The first author is partially supported by NSFC 12371053, the second author  is partially supported by  COST Action CaLISTA CA21109 (European Cooperation
in Science and Technology, www.cost.eu), and the third author is partially supported by NSFC 11531012}

\begin{abstract}
In this paper, we derive global bounds for the H\"older norm of the gradient 
of solutions of graphic mean curvature flow with boundary of arbitrary codimension.
\end{abstract}

\maketitle

\section{Introduction}
For an open set $\Om \subset \R^n$, the
 graph $u=(u^1,\cdots,u^m):\Om\to \R^{n+m}$ is minimal if $u$ satisfies a system of $m$ quasilinear elliptic equations (see \cite{O})
\begin{equation}
  \label{intro1}
  g^{ij}\p_{ij}u^\a=0\qquad \mathrm{in}\ \Om ,
\end{equation}
where $(g^{ij})$ is the inverse matrix of $g_{ij}=\de_{ij}+\sum_\a\p_i u^\a\p_j u^\a$. The Dirichlet problem is one of the classical problems in the field, that is, to find solutions with boundary data
\begin{equation}
  \label{intro2}
  u^\a=\psi^\a\qquad \mathrm{on}\ \p\Om
\end{equation}
for some given $\psi=(\psi^1,\cdots,\psi^m)$ on $\p\Om$. As it turns out, in order to obtain the existence and the regularity of solutions, some conditions on the geometry of $\p\Om$ and on the boundary data are needed.

Analogously, we can consider a time dependent version, the mean curvature flow.
Let $T$ be a positive constant.
Let  $u(x,t)=(u^1(x,t),\cdots,u^m(x,t))$ for $t\in(0,T)$, $x=(x_1,\cdots,x_n)\in\Om$,  and put
$U_t(x_1,\cdots,x_n)=(x_1,\cdots,x_n,u^1(x,t),\cdots,u^m(x,t))$. We consider the case where  $M_t=\mathrm{graph}_{u(\cdot,t)}=\{(x,u(x,t))|\, x\in\Om\}\subset\R^{n+m}$ moves along the mean curvature flow, i.e.,
$$\f{d U_t}{dt}=H_{M_t},$$
where $H_{M_t}$ denotes the mean curvature of $M_t$. In coordinates, $u$ satisfies the parabolic equations
\begin{equation}\label{intro2a}\aligned
&\f{\p u^\a}{\p t}=\f{d u^\a}{dt}-\f{\p u^\a}{\p x_j}\f{\p x_j}{\p t}\\
=&\f1{\sqrt{\det{g_{kl}}}}\p_i\left(g^{ij}\sqrt{\det{g_{kl}}}\partial_j u^\a\right)-\f{u^\a_j}{\sqrt{\det{g_{kl}}}}\p_i\left(g^{ij}\sqrt{\det{g_{kl}}}\right)=g^{ij}\partial_{ij}u^\a
\endaligned
\end{equation}
for each $\a=1,\cdots,m$ on $\Om\times(0,T)$, that is, the parabolic analogue of \eqref{intro1}. And we can then prescribe initial and boundary values. Obviously, the parabolic method provides also a possible approach to the original elliptic problem.

For codimension $m=1$, the elliptic problem is quite well understood from the classical paper \cite{JS} of Jenkins and Serrin. For higher codimension, that is, for $m>1$, the situation is more difficult and less well understood. A counterexample due to Lawson and Osserman in \cite{LO} tells us that the situation is fundamentally different from the case $m=1$. Of course, this implies similar difficulties for the parabolic case.

A crucial analytical step in the solution of the boundary value problems for $C^2$ data consists in deriving a global $C^{1,\g}$-estimate (for some $\g\in(0,1)$). An important step was taken by Thorpe \cite{T}, who showed 
that for $C^3$-boundary data on a bounded smooth domain, a solution with small $C^1$-norm satisfies a $C^{1,\g}$-estimate. It is more natural, however, to assume only a bound on the $C^2$-norm of the boundary data.
In \cite{DX}, the first author and the third author derived a global $C^{1,\g}$-estimate for minimal graphs of arbitrary codimension in Euclidean space under the $C^2$-norm of the boundary data
and some suitable conditions.

In this paper, we therefore derive  uniform $C^{1,\g}$-estimates for any solution $u$ to the graphic mean curvature flow with an assumption only on the $C^2$-norm of the boundary data,  provided the gradient $|Du|$ is bounded and the product of any two singular values of $Du$ is between -1 and 1 (see Theorem \ref{C1gMCF}). This condition on the product of any two singular values cannot be removed in view of  the counterexample of Lawson and Osserman \cite{LO} (in the elliptic case). Our estimates will also play an important role in \cite{DJX} and \cite{W}, where we provide general conditions for the existence of solutions to the Dirichlet problem.

The key ingredient
in the proof of Theorem \ref{C1gMCF} is to get the interior curvature estimates of mean curvature flow using Huisken's monotonicity formula \cite{Hui},
and the global $C^{1,\g}$-estimates using uniform parabolic equations \cite{Li} of Lieberman.

\section{Preliminary}

Let $\R^n$ be the standard $n$-dimensional Euclidean space.
For a point $\mathbf{x}=(x,t)\in\R^n\times\R=\R^{n+1}$, we set $|\mathbf{x}|=\max\{|x|,|t|^{1/2}\}$ and the cylinder
$$Q_R(\mathbf{x})=\left\{\mathbf{y}=(y,s)\in\R^{n+1}|\ |\mathbf{x}-\mathbf{y}|<R,\ s<t\right\}.$$
Denote $Q_R=Q_R(\mathbf{0})$ for short.
For a domain $W\subset\R^{n+1}$, we define the parabolic boundary $\mathcal{P}W$ to be the set of all points $\mathbf{x}\in\p W$ such that the cylinder $Q_\ep(\mathbf{x})$ contains points not in $W$ for any $\ep>0$.

Let us recall the standard H\"older norms for parabolic equations.
For every set $V\subset\R^{n+1}$, $\g_1\in(0,1]$, and every (vector-valued) function $f$ defined on $V$, we set
$$[f]_{\g_1;V}(\mathbf{x})=\sup_{\mathbf{y}\in V\setminus\{\mathbf{x}\}}\f{|f(\mathbf{y})-f(\mathbf{x})|}{|\mathbf{y}-\mathbf{x}|^{\g_1}}\qquad on\ \ V,$$
and $[f]_{\g_1;V}=\sup_{\mathbf{x}\in V}[f]_{\g_1;V}(\mathbf{x})$. For each $\g_2\in(0,2]$ and $\mathbf{x}=(x,t)\in V$, put
$$\lan f\ran_{\g_2;V}(\mathbf{x})=\sup_{(x,s)\in V\setminus\{\mathbf{x}\}}\f{|f(x,s)-f(\mathbf{x})|}{|s-t|^{\g_2/2}}\qquad on\ \ V,$$
and $\lan f\ran_{\g_2;V}=\sup_{\mathbf{x}\in V}\lan f\ran_{\g_2;V}(\mathbf{x})$.
Denote $|f|_{V}=\sup_{\mathbf{x}\in V}|f(\mathbf{x})|$.

Now for any $a>0$, we write $a=k+\g$ with a nonnegative integer $k$ and $\g\in(0,1]$. Let $D$ denote the spatial derivative and $\p_t$ denote the time derivative. Set
\begin{equation}\aligned
|f|_{a;V}(\mathbf{x})=\sum_{i+2j\le k}|D^i\p_t^jf|(\mathbf{x})+\sum_{i+2j= k}[D^i\p_t^jf]_{\g;V}(\mathbf{x})+\sum_{i+2j=k-1}\lan D^i\p_t^jf\ran_{\g+1;V}(\mathbf{x})
\endaligned\end{equation}
on $V$, and $|f|_{a;V}=\sup_{\mathbf{x}\in V}|f|_{a;V}(\mathbf{x})$.
We say $f\in H_{a}(V)$ if $|f|_{a;V}<\infty$.
For $\phi\in H_1(V)$, we define 2-dilation of $\phi$ (w.r.t. the spatial derivative)
$$\sup_V\big|\La^2D\phi\big|=\sup_{\mathbf{x}\in V}\big|\La^2D\phi(\mathbf{x})\big|=\sup_{\mathbf{x}\in V,1\le i<j\le n}\mu_i(\mathbf{x})\mu_j(\mathbf{x}),$$
where $\{\mu_k(\mathbf{x})\}_{k=1}^n$ are the singular values of $D\phi(\mathbf{x})$.

Let $\Om$ be an open set in $\R^n$, and $T$ is a positive constant.
A (vector-valued) function $f=(f^1,\cdots,f^m)$ is said to be in $C^2(\Om\times(0,T),\R^m)$, if each $f^\a$ is twice differentiable w.r.t. the variable $x\in\Om$, and each $f^\a$ is differentiable w.r.t. the variable $t\in(0,T)$.
Let $F_t$ be of the form
$F_t(x_1,\cdots,x_n)=(x_1,\cdots,x_n,f^1(x,t),\cdots,f^m(x,t))$
with $t\in(0,T)$, $x=(x_1,\cdots,x_n)\in\Om$ such that $M_t=\mathrm{graph}_{f(\cdot,t)}=\{(x,f(x,t))|\, x\in\Om\}\subset\R^{n+m}$ moves along the mean curvature flow, i.e.,
$$\f{d F_t}{dt}(x,t)=H_{M_t}(x),$$
where $H_{M_t}$ denotes the mean curvature of $M_t$. Let $\De_{M_t}$ denote the Laplacian of $M_t$.
From $\De F_t=H_{M_t}$, it follows that $f=(f^1,\cdots,f^m)$ satisfies the parabolic equations
\begin{equation}\aligned
&\f{\p f^\a}{\p t}=\f{d f^\a}{dt}-\f{\p f^\a}{\p x_j}\f{\p x_j}{\p t}\\
=&\f1{\sqrt{\det{g_{kl}}}}\p_i\left(g^{ij}\sqrt{\det{g_{kl}}}f^\a_j\right)-\f{f^\a_j}{\sqrt{\det{g_{kl}}}}\p_i\left(g^{ij}\sqrt{\det{g_{kl}}}\right)=g^{ij}f^\a_{ij}
\endaligned
\end{equation}
for each $\a=1,\cdots,m$ on $\Om\times(0,T)$,
where $g_{ij}=\de_{ij}+\sum_\a f^\a_if^\a_j$, and $(g^{ij})$ is the inverse matrix of $(g_{ij})$.

Let $L$ be the parabolic operator of the second order defined by
\begin{equation}
L\phi^\a=\f{\p \phi^\a}{dt}-g^{ij}_\phi\p_{ij}\phi^\a\qquad \mathrm{for}\ \a=1,\cdots,m,\\
\end{equation}
for each $\phi=(\phi^1,\cdots,\phi^m)\in C^2(\Om\times[0,T),\R^m)$,
where $(g^{ij}_\phi)$ is the inverse matrix of $(\de_{ij}+\sum_\a \p_i\phi^\a \p_j\phi^\a)$. For convenience, we denote $g^{ij}_\phi$ by $g^{ij}$.
We say $Lf=0$ if $Lf^\a=0$ for each $\a$.
$Lf=0$ implies that graph$_{f(\cdot,t)}$ moves by mean curvature flow.

\section{A priori H\"older gradient estimate for mean curvature flow}

\begin{lemma}\label{Int|A|}
For $R>0$, let $f=(f^1,\cdots,f^m)\in C^2(\overline{Q_R},\R^m)$ satisfy $Lf=0$ in $Q_R$ with $f(\mathbf{0})=0$, where $\mathbf{0}$ is the origin of $\R^n\times\R$. If $\sup_{Q_R}\left|\La^2Df\right|<1-\ep$ for some $\ep\in(0,1)$, then there is a constant $c=c(n,m,\ep,|Df|_{Q_R})$ depending only on $n,m,\ep,|Df|_{Q_R}$ such that
\begin{equation}
|D^2f|(\mathbf{0})\le cR^{-1}.
\end{equation}
\end{lemma}
\begin{proof}
By scaling, we only need to prove this Lemma with $R=1$. Put $Q=Q_1$ and $d_{Q}(\mathbf{x})=\inf_{\mathbf{y}\in\mathcal{P}Q}|\mathbf{x}-\mathbf{y}|$.
Let us prove it by contradiction.
Let $f_i$ be a sequence of smooth solutions of the  mean curvature flow in $Q$ with $f_i(\mathbf{0})=0\in\R^m$, $\sup\left|\La^2Df_i\right|\le1-\ep$ and $\limsup_i|Df_i|_Q<\infty$ such that
\begin{equation}\aligned\label{dQD2fi}
\lim_{i\rightarrow\infty}\left(\sup_{\mathbf{x}\in Q}d_{Q}(\mathbf{x})|D^2f_i(\mathbf{x})|\right)=\infty.
\endaligned
\end{equation}
Denote $R_i=\sup_{\mathbf{x}\in Q}d_{Q}(\mathbf{x})|D^2f_i(\mathbf{x})|$. There are points $\mathbf{x}_i=(x_i,t_i)\in Q$ such that $R_i=d_{Q}(\mathbf{x}_i)|D^2f_i(\mathbf{x}_i)|$.
Set
\begin{equation}\aligned
\widetilde{f_i}(x,t)=\f1{d_{Q}(\mathbf{x}_i)}f_i\left(d_{Q}(\mathbf{x}_i)x+x_i,d^2_{Q}(\mathbf{x}_i)t+t_i\right),
\endaligned
\end{equation}
then $\widetilde{f_i}$ still satisfies $L\widetilde{f_i}=0$ and $R_i=|D^2\widetilde{f_i}(\mathbf{0})|$. Moreover, $\sup\left|\La^2D\widetilde{f_i}\right|\le1-\ep$, $\limsup_i|D\widetilde{f_i}|_Q<\infty$ and
\begin{equation}\aligned
R_i=\sup_{\mathbf{x}=(x,t)\in Q_{d_Q(\mathbf{x}_i)}(\mathbf{x}_i)}\f{d_{Q}(\mathbf{x})}{d_{Q}(\mathbf{x}_i)}
|D^2\widetilde{f_i}|\bigg|_{\left(d^{-1}_Q(\mathbf{x}_i)(x-x_i),d^{-2}_Q(\mathbf{x}_i)(t-t_i)\right)}
=\sup_{\mathbf{y}\in Q}d_{Q}(\mathbf{y})|D^2\widetilde{f_i}(\mathbf{y})|.
\endaligned
\end{equation}

Let $\mathbf{B}_R$ denote the ball in $\R^{n+m}$ centered at the origin with  radius $R>0$.
Put $M^i_t=\mathrm{graph}_{\widetilde{f_i}(\cdot,t)}$.
Since $M^i_t$ is a Lipschitz graph with uniformly bounded Lipschitz constants,
\begin{equation}\aligned
\int_{M^i_{t}\cap \mathbf{B}_1}e^{\f{|X|^2}{4t}}
\endaligned
\end{equation}
is uniformly bounded independent of $i,t\in[-1,0)$. For each $t_{j}\in(0,1]$ with $t_{j}\rightarrow0$ as $j\rightarrow\infty$, there is a sequence $l_{i,j}\rightarrow\infty$ as $i\rightarrow\infty$ such that $\{l_{i,j}\}_i$ is a subsequence of $\{l_{i,j-1}\}_i$ for each $j\ge2$, and the limit
\begin{equation}\aligned
\lim_{i\rightarrow\infty}\int_{M^{l_{i,j}}_{-t_j}\cap \mathbf{B}_{1/2}}e^{-\f{|X|^2}{4t_j}}
\endaligned
\end{equation}
exists for any $j$. Up to the choice of the subsequence of $t_j,l_{i,j}$, we assume that the limit
\begin{equation}\aligned\label{tjlijexi}
t_j^{\f n2}\lim_{i\rightarrow\infty}\int_{M^{l_{i,j}}_{-t_j}\cap \mathbf{B}_{1/2}}e^{-\f{|X|^2}{4t_j}}
\endaligned
\end{equation}
exists (and is not equal to $\infty$) as $j\rightarrow\infty$. By Huisken's monotonicity formula \cite{Hui} (see also formula (7) in \cite{EH} or (1.2) in \cite{CM1} for example),
\begin{equation}\aligned
\int_{-t_j}^{-t_k}(-t)^{\f n2}&\left(\int_{M^{l_{i,j}}_{t}\cap \mathbf{B}_{1/2}}\left|H_{M^{l_{i,j}}_t}-\f{X}{2t}\right|^2e^{\f{|X|^2}{4t}}\right)dt
\le t_j^{\f n2}\int_{M^{l_{i,j}}_{-t_j}\cap \mathbf{B}_{1/2}}e^{-\f{|X|^2}{4t_j}}\\
&-t_k^{\f n2}\int_{M^{l_{i,j}}_{-t_k}\cap \mathbf{B}_{1/2}}e^{-\f{|X|^2}{4t_k}}+c_n\int_{-t_j}^{-t_k}(-t)^{\f n2}\left(\int_{M^{l_{i,j}}_{t}\cap \mathbf{B}_{1}}e^{\f{|X|^2}{4t}}\right)dt,
\endaligned
\end{equation}
where $X$ denotes the position vector in $\R^{n+m}$, and $c_n$ is a constant depending only on $n$.
Note that $M^i_t$ is a Lipschitz graph with a uniform Lipschitz constant. With \eqref{tjlijexi} we infer
\begin{equation}\aligned
\lim_{j\rightarrow\infty}\int_{-t_j}^{0}(-t)^{\f n2}\left(\lim_{i\rightarrow\infty}\int_{M^{l_{i,j}}_{t}\cap \mathbf{B}_{1/2}}\left|H_{M^{l_{i,j}}_t}-\f{X}{2t}\right|^2e^{\f{|X|^2}{4t}}\right)dt=0.
\endaligned
\end{equation}
There is a sequence $\{l_j\}$ with $l_j\in\{l_{i,j}\}_i$, such that
\begin{equation}\aligned\label{Mljt}
\lim_{j\rightarrow\infty}\int_{-t_j}^{0}(-t)^{\f n2}\left(\int_{M^{l_j}_{t}\cap \mathbf{B}_{1/2}}\left|H_{M^{l_j}_t}-\f{X}{2t}\right|^2e^{\f{|X|^2}{4t}}\right)dt=0
\endaligned
\end{equation}
and $\lim_{j\rightarrow\infty}R_{l_j}t_j=\infty$.

Set
\begin{equation}\aligned
\widehat{f_i}(x,t)=R_{l_i}\widetilde{f_{l_i}}\left(\f{x}{R_{l_i}},\f{t}{R_{l_i}^2}\right),
\endaligned
\end{equation}
and $\Si^i_t=\mathrm{graph}_{\widehat{f_i}(\cdot,t)}$. Then $t\in[-R^{2}_{l_i},0]\mapsto\Si^i_t$ is a sequence of mean curvature flows in $B_{R_{l_i}}(0)\times\R^m$ such that $\sup_{Q_{R_{l_i}}}\left|\La^2D\widehat{f_i}\right|\le1-\ep$, $\limsup_i|D\widehat{f_i}|_{Q_{R_{l_i}}}<\infty$ and
\begin{equation}\aligned
R_{l_i}=\sup_{\mathbf{x}\in Q_{R_{l_i}}}d_{Q_{R_{l_i}}}(\mathbf{x})|D^2\widehat{f_i}(\mathbf{x})|\to\infty.
\endaligned
\end{equation}
In particular, $|D^2\widehat{f_i}(\mathbf{x})|\le2$ on $Q_{R_{l_i}/2}$.
Hence, from \eqref{Mljt} we have
\begin{equation}\aligned\label{jinftyRljtjSilj}
\lim_{j\rightarrow\infty}\int_{-R_{l_j}^2t_j}^{0}(-t)^{\f n2}\left(\int_{\Si^{l_j}_{t}\cap \mathbf{B}_{R_{l_j}/2}}\left|H_{\Si^{l_j}_t}-\f{X}{2t}\right|^2e^{\f{|X|^2}{4t}}\right)dt=0.
\endaligned
\end{equation}
Since $\widehat{f_i}$ satisfies $L\widehat{f_i}=0$, then $|\p_t\widehat{f_i}(\mathbf{x})|\le2n$ from $|D^2\widehat{f_i}(\mathbf{x})|\le2$ on $Q_{R_{l_i}/2}$. By the  Arzela-Ascoli Theorem, we can assume that
$\hat{f_i}$ converges to $f_\infty$ on any bounded domain $K\subset Q_{R_{l_i}/2}$. Furthermore, $\sup_{Q_\infty}\left|\La^2Df_\infty\right|\le1-\ep$, $|Df_\infty|_{Q_\infty}<\infty$ and $|D^2f_\infty|_{Q_\infty}\le2$ with $Q_\infty=\lim_{R\rightarrow\infty}Q_R$. Denote $\Si^\infty_t=\mathrm{graph}_{f_\infty(\cdot,t)}$. By the Fatou Lemma, from \eqref{jinftyRljtjSilj} we conclude
\begin{equation}\aligned
\int_{-1}^{0}(-t)^{\f n2}\left(\int_{\Si^{\infty}_t\cap \mathbf{B}_{R}}\left|H_{\Si_\infty}-\f{X}{2t}\right|^2e^{\f{|X|^2}{4t}}\right)dt=0
\endaligned
\end{equation}
for any $R>0$. Hence $\Si^\infty_t$ are self-shrinkers (up to scalings) for all $t<0$. Therefore, they are smooth by Allard's regularity theorem. From \cite{DW}, $\Si^\infty_t$ is an $n$-plane for each $t$. Hence $\Si^i_t$ converges to $\Si^\infty_t$ smoothly (see \cite{Wh} for instance), but this contradicts to $|D^2\widehat{f_i}(\mathbf{0})|=1$. This suffices to complete the proof.
\end{proof}

\begin{lemma}\label{|Df|xi}
For each $R>0$, let $f=(f^1,\cdots,f^m)\in C^2(\overline{Q_R},\R^m)$ satisfy $Lf=0$ in $Q_R$ with $f(\mathbf{0})=0$. If $\sup_{Q_R}\left|\La^2Df\right|<1-\ep$ for some $\ep\in(0,1)$,
then there is a constant $c=c(n,m,\ep,|Df|_{Q_R})$ depending only on $n,m,\ep,|Df|_{Q_R}$ such that
for any $\xi\in\R^n\times\R^m$ and $\iota\in\R^m$
\begin{equation}
\sup_{Q_{R/2}}|Df-\xi|\le c\left(R^{-1}\sup_{\mathbf{x}\in Q_R}|f(\mathbf{x})-\xi\cdot x-\iota|+R\right).
\end{equation}
\end{lemma}
\begin{proof}
From Lemma \ref{Int|A|}, for any $\mathbf{x}\in Q_{R/2}$ we have
$$\sup_{Q_{R/3}(\mathbf{x})}|D^2f|\le cR^{-1},$$
where $c=c(n,m,\ep,|Df|_{Q_R})$ is a general constant depending only on $n,m,\ep,|Df|_{Q_R}$.
By contradiction and considering $\sup_{\mathbf{y}\in Q}d^2_{Q}(\mathbf{y})|D^3f_i(\mathbf{y})|$ instead of $\sup_{\mathbf{y}\in Q}d_{Q}(\mathbf{y})|D^2f_i(\mathbf{y})|$ in \eqref{dQD2fi}, it is not hard to get
$$\sup_{Q_{R/4}(\mathbf{x})}|D^3f|\le cR^{-2}.$$
Taking the derivative of the equation $Lf=0$, we have
\begin{equation}\label{Dptfc}
\sup_{Q_{R/4}(\mathbf{x})}|D\p_tf|\le cR^{-2}.
\end{equation}
Set 
$$g(\mathbf{x})=f(\mathbf{x})-\xi\cdot x-\iota$$ 
for each $\xi\in\R^n\times\R^m$ and $\iota\in\R^m$, then $Dg=Df-\xi$.
With an interpolation inequality (see Lemma 4.1 of \cite{Li} for instance), we have
\begin{equation}
R|Dg|(\mathbf{x})\le c\left(\sup_{Q_{R/4}(\mathbf{x})}|g|\right)^{\f{\g}{1+\g}}\left(\sup_{Q_{R/4}(\mathbf{x})}|g|+R^{1+\g}\sup_{Q_{R/4}(\mathbf{x})}[Dg]_{\g;Q_{R/4}(\mathbf{x})}\right)^{\f1{1+\g}}
\end{equation}
for any $\mathbf{x}\in Q_{R/2}$ and any $\g\in(0,1)$. Combining Lemma \ref{Int|A|} and \eqref{Dptfc}, we have
\begin{equation}
|Dg(\mathbf{x})|\le c\left(R^{-1}\sup_{\mathbf{y}\in Q_{R/4}(\mathbf{x})}|g(\mathbf{y})|+R\right),
\end{equation}
which suffices to complete the proof.
\end{proof}

Denote $B_r=B_r(0)\subset\R^n$, and $B_r^+=B_r\cap\R^n_+$ for short.
Let $\Om$ be a bounded domain in $\R^n$ with $\p\Om\in C^2$, and $\Om_T=\Om\times(0,T)$. Then its parabolic boundary is $\mathcal{P} \Om_T=(\overline{\Om}\times\{0\})\cup(\p\Om\times[0,T])$.
Let $\k_{1,\Om}(x), \cdots , \k_{n-1,\Om}(x)$ be
the principal curvatures of $\p\Om$ at each $x\in\p\Om$. Denote
$$\k_\Om=\max_{1\le i\le n,x\in\p\Om}|\k_{i,\Om}(x)|.$$
\begin{theorem}\label{C1gMCF}
For any $\ep>0$, $T>0$ and $\psi=(\psi^1,\cdots,\psi^m)\in C^{2}\left(\overline{\Om},\R^m\right)$,
there are constants $\g\in(0,1)$, $C>0$ depending only on $n,m$, $\ep$, $|Df|_{\Om_T}$, $|\psi|_{2,\Om}$, $\k_\Om$ and $\mathrm{diam}\,\Om$ such that
if $f=(f^1,\cdots,f^m)\in H_{2}\left(\Om_T\right)\cap H_{1+\g}\left(\overline{\Om_T}\right)$ satisfies $Lf=0$ in $\Om_T$ with $f(\cdot,0)=\psi$ on $\Om\times\{0\}$, $f(\cdot,t)=\psi$ on $\p\Om$ for each $t\in[0,T]$, and $\sup_{\Om_T}\left|\La^2Df\right|<1-\ep$, then $[Df]_{\g;\Om_T}\le C$.
\end{theorem}
\begin{proof}
From Lemma \ref{Ex-MCF} in Appendix I, the flow $Lf=0$ has the short-time existence.
Hence, there are constants $\ep_*>0$ and $c>0$ depending only on $n,m,\k_\Om,\mathrm{diam}\,\Om,|\psi|_{2,\Om}$ such that
\begin{equation}\aligned\label{f32epT}
|f|_{\f32;\Om_{\min\{\ep_*,T\}}}\le c.
\endaligned
\end{equation}
In order to finish the proof, we only need to consider the case $T\ge\ep_*$.
We shall first derive H\"older estimates for $Df$ on $\p\Om\times(\ep_*/2,T)$ by following the idea of the proof of Theorem 12.5 in \cite{Li}.

For any $\mathbf{x_*}=(x_*,t_*)\in\p\Om\times(\ep_*/2,T)$, up to a translation we may assume that $\mathbf{x_*}$ is the origin in $\R^n\times\R$, and $f$ is defined in $\Om\times(-t_*,T-t_*)$.
Let
$$Q_r^+=Q_r\cap\{(x',x_n)\in\R^{n-1}\times\R|\, x_n>0\}$$
for each $r>0$. Let $F$ be the map defined in Appendix II with its inverse $F^{-1}$, and we choose $r_0=\min\{\k_\Om,\sqrt{\ep_*/2}\}$ in Appendix II. Let
$$\hat{f}=f\circ F^{-1}-\psi\circ F^{-1}$$
be the function defined in Appendix II. Then $\hat{f}=0$ in $\p Q_{r_0}^+\cap Q_{r_0}$.
Let 
$$\hat{L}=\p_t-G^{kl}(y,D\hat{f})\p_{kl},$$ 
then from \eqref{Gklu} one has
\begin{equation}\aligned
\hat{L}\hat{f}=\Th(y,D\hat{f}(y)) \qquad \mathrm{in}\ Q^+_{r_0},
\endaligned
\end{equation}
where the functions $G^{kl}$ and $\Th$ satisfy \eqref{CondGTh}.
Let
$$G(\r,R)=\left\{(x',x_n,t)\in\R^{n-1}\times\R\times\R\big|\, |x'|<R,\, 0<x_n<\r R,\, -R^2<t<0\right\}$$
for all constants $\r,R>0$.
For any set $V\in\R^{n+1}$ and any function $\varphi$ on $V$, denote $\mathrm{osc}_V\varphi=\sup_V\varphi-\inf_V\varphi$.
Denote $\mathbf{y}=(y',y_n,s)\in\R^{n-1}\times\R\times\R$.
From Lemma 7.47 in \cite{Li}, there are constants $\g',\r_*\in(0,\f12]$, and a constant $c'$ depending only on $n,m$, $|Df|_{\Om_T}$, $|\psi|_{2;\Om_T}$ and $\k_\Om$ such that
\begin{equation}\aligned\label{oscGrrfa}
\mathrm{osc}_{G(\r_*,r)}\f{\hat{f}^\a(\mathbf{y})}{y_n}\le c'\left(\f rR\right)^{\g'}\left(\mathrm{osc}_{G(\r_*,R)}\f{\hat{f}^\a(\mathbf{y})}{y_n}+R\right)
\endaligned
\end{equation}
for each $\a=1,\cdots,m$, and all $0<r<R\le r_0$.

For any fixed $\mathbf{x}=(x,t)=(x',x_n,t)\in Q_r^+$ with $r\le \f12\r_*r_0$, and $1\le\a\le m$, put $\mathbf{x}'=(x',0,t)$, $\z=D\hat{f}^\a(\mathbf{x}')$ and $\z_n=\p_{x_n}\hat{f}^\a(\mathbf{x}')$. By the definition of $\hat{f}$, $\lan\z,y\ran=\z_ny_n$ for any $y=(y_1,\cdots,y_n)\in\R^n$.
We choose $R=r_0$ in \eqref{oscGrrfa}, then
\begin{equation}\label{oscGrrfa*}
\sup_{\mathbf{y}\in Q_r^+}\left|\z_n-\f{\hat{f}^\a(\mathbf{y})}{y_n}\right|\le Cr^{\g'}
\end{equation}
for all $0<r<\f12\r_*r_0$,
which implies
\begin{equation}\label{hatfxiy}
\sup_{\mathbf{y}\in Q_{x_n}(\mathbf{x})}\left|\hat{f}^\a(\mathbf{y})-\lan\z,y\ran\right|\le Cx_nr^{\g'}.
\end{equation}
Denote $F(\mathbf{y})=(F(y),t_y)$ and $\mathbf{y}=F^{-1}(F(y),t_y)$ for each $\mathbf{y}=(y,t_y)$.
From Lemma \ref{|Df|xi}, there is a general constant $C$ depending only on $n,m$, $\ep$, $\mathrm{diam}\,\Om$, $|Df|_{\Om_T}$, $|\psi|_{2;\Om_T}$ and $\k_\Om$ such that
\begin{equation}\aligned\label{Dfpsi**}
&\left|Df^\a\big|_{F^{-1}(\mathbf{x})}-D\psi^\a\big|_{F^{-1}(x)}-(DF)^T\big|_{F^{-1}(x)}\z\right|\\
\le& \f C{\de x_n}\sup_{\mathbf{y}\in Q_{\de x_n}(F^{-1}(\mathbf{x}))}\left|f^\a(\mathbf{y})-\lan D\psi^\a\big|_{F^{-1}(x)},y\ran-\lan\z,DF\big|_{F^{-1}(x)}y\ran-\lan\z,x\ran\right|+C\de x_n.
\endaligned\end{equation}
Here, $\de$ is a positive constant $\le1$ to be defined later. The bound of $|\psi|_{2;\Om_T}$ implies
\begin{equation}\aligned\label{psi1}
\left|\psi^\a(y)-\lan D\psi^\a\big|_{F^{-1}(x)},y\ran\right|\le C\de^2x_n^2
\endaligned\end{equation}
for all $y\in B_{\de x_n}(F^{-1}(x))$. By the definition of $F$,
\begin{equation}\aligned\label{z1}
\left|\lan \z,F(y)\ran-\lan\z,DF\big|_{F^{-1}(x)}y\ran-\lan\z,x\ran\right|\le C\de^2x_n^2
\endaligned\end{equation}
for all $y\in B_{\de x_n}(F^{-1}(x))$.
Denote $\mathbf{z}=(z,s')$ for some $s'\in\R$.
Combining the definition of $\hat{f}$ in Appendix II and \eqref{Dfpsi**}\eqref{psi1}\eqref{z1}, we conclude that
\begin{equation}
\left|(DF)^T\big|_{F^{-1}(x)}D\hat{f}^\a\big|_{\mathbf{x}}-(DF)^T\big|_{F^{-1}(x)}\z\right|\le \f C{\de x_n}\sup_{\mathbf{z}\in F(Q_{\de x_n}(F^{-1}(\mathbf{x})))}\left|\hat{f}^\a(\mathbf{z})-\lan\z,z\ran\right|+C\de x_n.
\end{equation}
We choose a suitably small $\de>0$ depending on $\k_\Om$ such that $Q_{\de x_n}(F^{-1}(\mathbf{x}))\subset F^{-1}(Q_{x_n}(\mathbf{x}))$. Then it follows that
\begin{equation}\label{Dhatfxxi}
|D\hat{f}^\a(\mathbf{x})-\z|\le \f C{x_n}\sup_{\mathbf{y}\in Q_{x_n}(\mathbf{x})}\left|\hat{f}^\a(\mathbf{y})-\lan\z,y\ran\right|+Cx_n.
\end{equation}
Combining the two inequalities \eqref{hatfxiy}\eqref{Dhatfxxi} and $r\in(0,1]$ yields
\begin{equation}\label{Qxnhatf}
\left|D\hat{f}^\a(\mathbf{x})-D\hat{f}^\a(\mathbf{x}')\right|\le Cr^{\g'}
\end{equation}
for all $\mathbf{x}\in Q_r^+$ and $0<r<\f12\r_*r_0$.
From \eqref{oscGrrfa*}, it is clear that
\begin{equation}\label{hatf0x'}
\left|D\hat{f}^\a(\mathbf{0})-D\hat{f}^\a(\mathbf{x}')\right|\le Cr^{\g'}.
\end{equation}
With \eqref{Qxnhatf} and \eqref{hatf0x'}, it follows that
\begin{equation}
\left|D\hat{f}^\a(\mathbf{x})-D\hat{f}^\a(\mathbf{0})\right|\le Cr^{\g'}.
\end{equation}
Hence we have deduced the uniform $C^{1,\g'}$-norm of $f$ on $\p\Om\times(\ep_*/2,T)$.
Namely, for any $\mathbf{x}\in \p\Om\times(\ep_*/2,T)$ and $\mathbf{y}\in \Om\times(\ep_*/2,T)$, there holds
\begin{equation}\label{Holdfep*}
\left|Df(\mathbf{x})-Df(\mathbf{y})\right|\le C|\mathbf{x}-\mathbf{y}|^{\g'}.
\end{equation}

Let $\mathbf{z}\in \Om\times(\ep_*/2,T)$.
For $|\mathbf{y}-\mathbf{z}|\le \min\{d(\mathbf{y})^2,d(\mathbf{y})/2\}$, from Lemma \ref{Int|A|} we have
\begin{equation}
\left|Df(\mathbf{z})-Df(\mathbf{y})\right|\le \f{C}{d(\mathbf{y})}|\mathbf{z}-\mathbf{y}|\le C|\mathbf{z}-\mathbf{y}|^{\f12}.
\end{equation}
For $|\mathbf{y}-\mathbf{z}|\ge \min\{d(\mathbf{y})^2,d(\mathbf{y})/2\}$, let $\mathbf{y_*}$ be a point in $\p\Om\times(\ep_*/2,T)$ such that $|\mathbf{y}-\mathbf{y_*}|=d(\mathbf{y})$. Then
$$|\mathbf{y_*}-\mathbf{z}|\le |\mathbf{y_*}-\mathbf{y}|+|\mathbf{y}-\mathbf{z}|= d(\mathbf{y})+|\mathbf{y}-\mathbf{z}|.$$
Combining \eqref{Holdfep*} we have
\begin{equation}\aligned
&\left|Df(\mathbf{z})-Df(\mathbf{y})\right|\le \left|Df(\mathbf{y_*})-Df(\mathbf{y})\right|+\left|Df(\mathbf{z})-Df(\mathbf{y_*})\right|\\
\le& C|\mathbf{y_*}-\mathbf{y}|^{\g'}+C|\mathbf{y_*}-\mathbf{z}|^{\g'}\le Cd(\mathbf{y})^{\g'}+C\left(d(\mathbf{y})+|\mathbf{y}-\mathbf{z}|\right)^{\g'}\le C|\mathbf{y}-\mathbf{z}|^{\g'/2}.
\endaligned\end{equation}
Hence, $[Df]_{\g'/2;\Om\times(\ep_*/2,T)}\le C$. Together with \eqref{f32epT}, we deduce
\begin{equation}\aligned
\left[Df\right]_{\g'/2;\Om_{T}}\le C.
\endaligned
\end{equation}
This completes the proof.
\end{proof}
\begin{remark}
If we remove the condition $\sup_{\Om_T}\left|\La^2Df\right|<1-\ep$ in Theorem \ref{C1gMCF}, then it's almost impossible to control the Hessian of $f$.
In general, $\sup_{\Om}\left|\La^2Df(\cdot,0)\right|<1-\ep$ does not preserved along mean curvature flow.
However, in \cite{DJX} we find a class of parabolic boundary data $\psi$ such that $\sup_{\Om}\left|\La^2Df(\cdot,t)\right|<1-\ep$ does preserve in a suitable sense along mean curvature flow.
\end{remark}
\begin{remark} Under the assumption of Theorem \ref{C1gMCF}, we can use the conclusion of Theorem \ref{C1gMCF} and Theorem 5.15 in \cite{Li} to deduce that for any $\g\in(0,1)$
there is a constant $C>0$ depending only on $n,m$, $\ep,\g$, $|Df|_{\Om_T}$, $|\psi|_{2,\Om}$, $\k_\Om$, $\mathrm{diam}\,\Om$ and $T$ such that
$$|f|_{1+\g;\Om_T}\le C.$$
\end{remark}

\section{Appendix I}

Let $\Om$ be a bounded domain in $\R^n$ with $\p\Om\in C^2$, and $\psi=(\psi^1,\cdots,\psi^m)\in C^2(\overline{\Om},\R^m)$.
Denote $\Om_T=\Om\times(0,T)$ for some $T>0$.
Let us consider the flow
\begin{equation}\label{MCF*}
\left\{\begin{split}
Lf^\a=&\f{\p f^\a}{dt}-g^{ij}f^\a_{ij}=0\qquad \mathrm{in}\ \Om_T\\
f^\a=&\psi^\a\qquad\qquad\qquad\quad \mathrm{on}\ \mathcal{P}\Om_T\\
\end{split}\right.\qquad\qquad \mathrm{for}\ \a=1,\cdots,m,
\end{equation}
where $(g^{ij})$ is the inverse matrix of $g_{ij}=\de_{ij}+\sum_\a \p_if^\a \p_jf^\a$.
In general, $L\psi\neq0$ on $\p\Om\times(0,T)$. Hence, we do not have the standard boundary estimate or the short-time existence of \eqref{MCF*} immediately.
Now let us define certain weighted norms as follows (see page 47 in \cite{Li}).
For each $\mathbf{x}=(x,t)\in\R^n\times\R$, let
$$\r(\mathbf{x})=\inf\{|\mathbf{y}-\mathbf{x}||\ \mathbf{y}=(y,s)\in\mathcal{P}\Om_T,\, s<t\},$$
and $\r(\mathbf{x},\mathbf{y})=\min\{\r(\mathbf{x}),\r(\mathbf{y})\}$. Denote $\mathrm{diam}\, \Om_T=\sup_{\mathbf{x},\mathbf{y}\in \Om_T}|\mathbf{x}-\mathbf{y}|$. For any (vector-valued) function $\phi$ on $\Om_T$, we define
\begin{equation}
|\phi|_{0;\Om_T}^{(b)}=\left\{\begin{split}
\sup_{\Om_T} \r^b|\phi|\qquad \mathrm{for}\ b\ge0\\
(\mathrm{diam}\, \Om_T)^b\sup_{\Om_T}|\phi|\qquad \mathrm{for}\ b<0\\
\end{split}\right..
\end{equation}
We further assume $\phi\in C^{2}(\Om_T)$.
For $a=k+\g>0$ with $\g\in(0,1]$ and $a+b\ge0$, we define
\begin{equation}\aligned
\left[ \phi \right]^{(b)}_{a;\Om_T}=&\sup_{\mathbf{x}\neq \mathbf{y}\in \Om_T}\r(\mathbf{x},\mathbf{y})^{a+b}\sum_{i+2j=k}|\mathbf{x}-\mathbf{y}|^{-\g}|D^i\p^j_t\phi(\mathbf{x})-D^i\p^j_t\phi(\mathbf{y})|,\\
\left\lan \phi \right\ran^{(b)}_{a;\Om_T}=&\sup_{\mathbf{x}=(x,t)\neq \mathbf{y}=(x,s)\in \Om_T}\r(\mathbf{x},\mathbf{y})^{a+b}\sum_{i+2j=k-1}|\mathbf{x}-\mathbf{y}|^{-1-\g}|D^i\p^j_t\phi(\mathbf{x})-D^i\p^j_t\phi(\mathbf{y})|,\\
\left|\phi \right|^{(b)}_{a;\Om_T}=&\sum_{i+2j\le k}|D^i\p^j_t\phi|_0^{(i+2j+b)}+\left[ \phi \right]^{(b)}_a+\left\lan \phi \right\ran^{(b)}_a.
\endaligned
\end{equation}

Now let us state the short-time existence of mean curvature flows (see Theorem 8.2 in \cite{Li}).
\begin{lemma}\label{Ex-MCF}
For each $0<\g<1$ and $\psi=(\psi^1,\cdots,\psi^m)\in H_{1+\g}(\overline{\Om})$, there are constants $\de>0$ and $c>0$ depending only on $n,m,\g,\k_\Om,\mathrm{diam}\,\Om,|\psi|_{1+\g,\Om}$ and a function $f=(f^1,\cdots,f^m)\in C^\infty(\Om_\de,\R^m)\cap H_{1+\g}(\overline{\Om_\de})$ with $Lf=0$ in $\Om_\de$ such that $f(\cdot,0)=\psi$ on $\Om\times\{0\}$, $f(\cdot,t)=\psi$ on $\p\Om$ for each $t\in[0,\de]$, and $|f|_{1+\g;\Om_\de}\le c$.
\end{lemma}
\begin{proof}
For any $\de\in(0,\mathrm{diam}\,\Om)$ and $0<\th<\g$, let $C_0=1+|\psi|_{1+\th,\Om}$ and
$$\mathfrak{B}=\{\phi=(\phi^1,\cdots,\phi^m)\in H_{1+\th}(\overline{\Om_\de})|\, |\phi|_{1+\th;\Om_\de}\le C_0\}.$$
Denote $T=\mathrm{diam}\,\Om$. We extend $\phi=(\phi^1,\cdots,\phi^m)$ to be a (vector-valued) function in $H_{1+\th}(\overline{\Om_T})$ by $\phi(\cdot,t)=\phi(\cdot,\de)$ for all $t\in(\de,T]$. It follows that $|\phi|_{1+\th;\Om_T}\le|\phi|_{1+\th;\Om_\de}\le C_0$.
Let $a^\phi_{ij}=\de_{ij}+\sum_\a \p_i\phi^\a \p_j\phi^\a$ for each $i,j=1,\cdots,n$, and $(a_\phi^{ij})$ be the inverse matrix of $(a^\phi_{ij})$.
Then $|a_\phi^{ij}|_{\th;\Om_T}$ is bounded by a constant depending only on $n,m,C_0$.

For every function $\varphi\in C^{2}(\Om_T)$, we define
\begin{equation}\aligned
L_\phi \varphi=\f{\p \varphi}{dt}-a_\phi^{ij}\p_{ij}\varphi.
\endaligned
\end{equation}
Let $\psi(x,t)=\psi(x,0)$ for each $(x,t)\in\overline{\Om_T}$.
From Theorem 5.15 in \cite{Li}, for each $\a$ there is a unique function $\varphi^\a$ satisfying $L_\phi \varphi^\a=0$ with $\varphi^\a=\psi^\a$ on $\mathcal{P}\Om_T$.
Moreover, there is a general constant $c$ depending on $n,m,\g,\k_\Om,\mathrm{diam}\,\Om$ and $|\phi|_{1+\th;\Om_T}$ such that
\begin{equation}\aligned\label{fathg}
|\varphi^\a|_{2+\th;\Om_T}^{(-1-\g)}\le c|\psi^\a|_{1+\g,\Om}.
\endaligned
\end{equation}
In particular, $(\mathrm{diam}\, \Om_T)^{-\g}\sup_{\Om_T}|D\varphi^\a|\le c|\psi^\a|_{1+\g,\Om}$.
For any $\mathbf{x},\mathbf{y}\in\Om_T$, without loss of generality, we assume $\r(\mathbf{y})\le \r(\mathbf{x})$.
We assume $\r(\mathbf{y})<\f12\min\{|\mathbf{x}-\mathbf{y}|,\k_{\Om}\}$, or else
it's clear that 
\begin{equation}\aligned
|D\varphi^\a(\mathbf{y})-D\varphi^\a(\mathbf{x})|\le c'|\varphi^\a|_{2+\th;\Om_T}^{(-1-\g)}|\mathbf{y}-\mathbf{x}|^\g,
\endaligned
\end{equation}
where $c'$ is a general constant depending only on $n,m,\g,\k_\Om,\mathrm{diam}\,\Om$.
There exists a point $\mathbf{z_y}\in\Om_T$ such that $|\mathbf{z_y}-\mathbf{y}|=|\mathbf{y}-\mathbf{x}|$ and $\r(\mathbf{z_y})=\r(\mathbf{y})+|\mathbf{y}-\mathbf{z_y}|$.
Then
\begin{equation}\aligned\label{zyxy}
|\mathbf{x}-\mathbf{y}|\le\r(\mathbf{z_y})<\f12|\mathbf{x}-\mathbf{y}|+|\mathbf{y}-\mathbf{z_y}|\le\f32|\mathbf{x}-\mathbf{y}|.
\endaligned
\end{equation}
We choose a sequence of points $\mathbf{y}_0=\mathbf{y},\mathbf{y}_1,\cdots,\mathbf{y}_{N-1},\mathbf{y}_N=\mathbf{y_z}$
such that $|\mathbf{y}_i-\mathbf{y}_{i+1}|=\f12\r(\mathbf{y}_{i+1})$ and $\r(\mathbf{y}_{i+1})=\r(\mathbf{y}_{i})+|\mathbf{y}_i-\mathbf{y}_{i+1}|$ for $i=0,\cdots,N-1$,
and $|\mathbf{y}_{N-1}-\mathbf{y}_{N}|\le\f12\r(\mathbf{y}_{N})$, $\r(\mathbf{y}_{N})=\r(\mathbf{y}_{N-1})+|\mathbf{y}_{N-1}-\mathbf{y}_{N}|$.
Then $\r(\mathbf{y}_{i+1})=2\r(\mathbf{y}_{i})$ for $i=0,\cdots,N-1$, and $\r(\mathbf{y}_{N})\le2\r(\mathbf{y}_{N-1})$.
Hence from \eqref{zyxy} one has
\begin{equation}\aligned\label{yizy}
\r(\mathbf{y}_{i})=\f12\r(\mathbf{y}_{i+1})=2^{i+1-N}\r(\mathbf{y}_{N-1})\le2^{i+1-N}\r(\mathbf{z_y})\le3\times2^{i-N}|\mathbf{x}-\mathbf{y}|.
\endaligned
\end{equation}
Since
\begin{equation}\aligned
|D\varphi^\a(\mathbf{y}_{i+1})-D\varphi^\a(\mathbf{y}_i)|\le
\left(\f{\r(\mathbf{y}_{i+1})}2\right)^{1+\th}\f{|D\varphi^\a(\mathbf{y}_{i+1})-D\varphi^\a(\mathbf{y}_i)|}{|\mathbf{y}_{i+1}-\mathbf{y}_i|^{1+\th}}
\le\r(\mathbf{y}_{i+1})^\g|\varphi^\a|_{2+\th;\Om_T}^{(-1-\g)},
\endaligned
\end{equation}
combining this with \eqref{yizy} we get
\begin{equation}\aligned\label{Dfayzy}
&|D\varphi^\a(\mathbf{y})-D\varphi^\a(\mathbf{z_y})|\le\sum_{i=0}^{N-1}|D\varphi^\a(\mathbf{y}_{i+1})-D\varphi^\a(\mathbf{y}_i)|\\
\le& 3|\mathbf{x}-\mathbf{y}|^\g|\varphi^\a|_{2+\th;\Om_T}^{(-1-\g)}\sum_{i=0}^{N-1}2^{(i+1-N)\g}\le \f3{1-2^{-\g}}|\mathbf{x}-\mathbf{y}|^\g|\varphi^\a|_{2+\th;\Om_T}^{(-1-\g)}.
\endaligned
\end{equation}

For $\r(\mathbf{x})\ge\f12\min\{|\mathbf{x}-\mathbf{y}|,\k_{\Om}\}$, we clearly have
\begin{equation}\aligned\label{Dfaxzy}
|D\varphi^\a(\mathbf{x})-D\varphi^\a(\mathbf{z_y})|\le c'|\mathbf{x}-\mathbf{y}|^\g|\varphi^\a|_{2+\th;\Om_T}^{(-1-\g)}.
\endaligned
\end{equation}
For $\r(\mathbf{x})<\f12\min\{|\mathbf{x}-\mathbf{y}|,\k_{\Om}\}$, there exists a point $\mathbf{z_x}\in\Om_T$ such that $|\mathbf{z_x}-\mathbf{x}|=|\mathbf{y}-\mathbf{x}|$ and $\r(\mathbf{z_x})=\r(\mathbf{x})+|\mathbf{x}-\mathbf{z_x}|$.
Then analogously to the above argument, it follows that
\begin{equation}\aligned\label{xzx11}
|D\varphi^\a(\mathbf{x})-D\varphi^\a(\mathbf{z_x})|\le c'|\mathbf{x}-\mathbf{y}|^\g|\varphi^\a|_{2+\th;\Om_T}^{(-1-\g)}.
\endaligned
\end{equation}
Combining the definition of $\mathbf{z_x},\mathbf{z_y}$, 
$$|\mathbf{x}-\mathbf{y}|\le\min\{\r(\mathbf{z_x}),\r(\mathbf{z_y})\}\le\max\{\r(\mathbf{z_x}),\r(\mathbf{z_y})\}\le\f32|\mathbf{x}-\mathbf{y}|$$ and then
\begin{equation}\aligned\label{xzxy11}
|D\varphi^\a(\mathbf{z_x})-D\varphi^\a(\mathbf{z_y})|\le c'|\mathbf{x}-\mathbf{y}|^\g|\varphi^\a|_{2+\th;\Om_T}^{(-1-\g)}.
\endaligned
\end{equation}
Combining \eqref{xzx11} and \eqref{xzxy11}, we get \eqref{Dfaxzy}.
With \eqref{Dfayzy}\eqref{Dfaxzy}, we deduce
\begin{equation}\aligned
\left[D\varphi^\a\right]_{\g;\Om_T}\le c'|\varphi^\a|_{2+\th;\Om_T}^{(-1-\g)}.
\endaligned
\end{equation}
Similarly, we have
\begin{equation}\aligned
\left\lan \varphi^\a\right\ran_{1+\g;\Om_T}\le c'|\varphi^\a|_{2+\th;\Om_T}^{(-1-\g)}.
\endaligned
\end{equation}
Therefore, with \eqref{fathg} we get
\begin{equation}\aligned
\left| \varphi^\a\right|_{1+\g;\Om_T}\le c'|\varphi^\a|_{2+\th;\Om_T}^{(-1-\g)}\le c|\psi^\a|_{1+\g,\Om}.
\endaligned
\end{equation}
From the Newton-Leibniz formula,
$$|\varphi^\a-\psi^\a|\le |\varphi^\a-\psi^\a|_{1+\g;\Om_\de}\de^{\f{1+\g}2}\le c|\psi^\a|_{1+\g,\Om}\de^{\f{1+\g}2}$$
on $\Om_\de$ and then
$$|\varphi^\a-\psi^\a|_{1+\th;\Om_\de}\le c|\varphi^\a-\psi^\a|_{\Om_\de}^{\f{\g-\th}{1+\g}}|\varphi^\a-\psi^\a|_{1+\g;\Om_\de}^{\f{1+\th}{1+\g}}\le c\de^{\f{\g-\th}2}|\psi^\a|_{1+\g,\Om_\de}^{\f{1+\th}{1+\g}}$$
by interpolation (see Proposition 4.2 in \cite{Li} for instance).
Hence, $|\varphi|_{1+\th;\Om_\de}\le C_0$ for the sufficiently small $\de>0$.

Now we define a map $J:\, \mathfrak{B}\rightarrow H_{1+\th}(\overline{\Om_\de})$ by $\varphi=(\varphi^1,\cdots,\varphi^m)=J\phi$ (restricted on $\overline{\Om_\de}$).
Then $|\varphi|_{1+\th;\Om_\de}\le C_0$ implies $J\mathfrak{B}\subset\mathfrak{B}$.
Since $\mathfrak{B}$ is a convex compact subset of $H_{1+\th'}(\overline{\Om_\de})$ for any $\th'\in(0,\th)$,
it follows that $J$ has a fixed point $f=(f^1,\cdots,f^m)\in H_{1+\g}(\overline{\Om_\de})$ with
\begin{equation}\aligned
\left| f\right|_{1+\g;\Om_\de}\le c|\psi|_{1+\g,\Om}.
\endaligned
\end{equation}
This completes the proof.
\end{proof}

\section{Appendix II}

For studying the boundary regularity of parabolic systems, we usually only need to consider a similar system on a portion of a half space by a coordinate transformation.
Let $B_r$ be a ball with radius $r$ and centered at the origin in $\R^n$. Let $\Om$ be a domain in $\R^n$ with $C^2$-boundary $\p\Om\ni0$.
For any $r_0\in(0,1/\k_\Om)$, we assume that there is a coordinate change $F:\ B_{r_0}\rightarrow F(B_{r_0})\subset\R^n$ such that $F,F^{-1}$ are $C^2$-maps satisfying $F(B_{r_0}\cap\p\Om)\subset\{y=(y_1,\cdots,y_n)\in\R^n |\,y_n=0\}$, $F(B_{r_0}\cap\Om)\subset\{y=(y_1,\cdots,y_n)\in\R^n |\, y_n>0\}$, and the matrix $DF(DF)^T$ has eigenvalues between two constants $\La_F^{-1}$ and $\La_F$, where $\La_F>1$ is a constant depending only on $n,\k_\Om$. Without loss of generality, we can assume
\begin{equation}\aligned
\sup_{B_{r_0}}|D^2F|\le\La_F.
\endaligned
\end{equation}
For each $C^2$ vector-valued function $f=(f^1,\cdots,f^m)$ in $\Om_{t_1t_2}=\Om\times(t_1,t_2)$,
we define a new function $\tilde{f}$ by
$\tilde{f}(y,t)=f(x,t)$ with $y=(y_1,\cdots,y_n)=F(x)=(F^1(x),\cdots,F^n(x))$. Then $Df=DF\cdot D\tilde{f}$. Put
\begin{equation}\aligned
A_{ij}(y,D\tilde{f}(y,t))=\de_{ij}+\p_{y_k}\tilde{f}^\a(y,t)\p_{x_i}F^k(F^{-1}(y))\cdot\p_{y_l}\tilde{f}^\a(y,t)\p_{x_j}F^l(F^{-1}(y)).
\endaligned
\end{equation}

Now we assume that $f$ satisfies the flow
\begin{equation}\aligned
\f{\p f}{\p t}-g^{ij}\p^2_{x_ix_j}f=0\qquad \mathrm{in}\ \  \Om_{t_1t_2}
\endaligned
\end{equation}
with $f=\psi$ on $\mathcal{P}\Om_{t_1t_2}$,
where $(g^{ij})$ is the inverse matrix of $g_{ij}=\de_{ij}+\sum_\a f^\a_if^\a_j$.
Then
\begin{equation}\aligned
0=\p_t\tilde{f}-\p_{x_i}F^kA^{ij}\p_{x_j}F^l\p^2_{y_ky_l}\tilde{f}-A^{ij}\p^2_{x_ix_j}F^k\p_{y_k}\tilde{f}.
\endaligned
\end{equation}
Set $\tilde{\psi}$ by $\psi=\tilde{\psi}\circ F$, and $\hat{f}=\tilde{f}-\tilde{\psi}$ so that $\hat{f}=0$ on $F(\p\Om\cap B_{r_0})\times(t_1,t_2)$. Put
\begin{equation}\aligned
&G^{kl}(y,D\hat{f})=A^{ij}(y,D\hat{f}+D\tilde{\psi})\p_{x_i}F^k\p_{x_j}F^l,\\
&\Th(y,D\hat{f})=A^{ij}(y,D\hat{f}+D\tilde{\psi})\p^2_{x_ix_j}F^k\left(\p_{y_k}\hat{f}+\p_{y_k}\tilde{\psi}\right)+G^{kl}(y,D\hat{f})\p^2_{y_ky_l}\tilde{\psi}-\p_t\tilde{\psi}.
\endaligned
\end{equation}
Then $\hat{f}$ satisfies the parabolic system
\begin{equation}\aligned\label{Gklu}
\p_t\hat{f}=G^{kl}(y,D\hat{f}(y))\p^2_{y_ky_l}\hat{f}+\Theta(y,D\hat{f}(y)).
\endaligned
\end{equation}
Hence there is a positive constant $\la_f$ depending only on $n,m$, $\La_F$, $|Df|_{0;\Om_{t_1t_2}}$ and $|D\psi|_{0;\Om_{t_1t_2}}$ such that
\begin{equation}\aligned\label{CondGTh}
\la_f^{-1}I_n\le (G^{kl})\le \La_F I_n,\qquad |\Th|\le c_n\la_f|\psi|_{2;\Om_{t_1t_2}}
\endaligned
\end{equation}
on $F(\Om\cap B_{r_0})\times(t_1,t_2)$. Here, $c_n$ is a positive constant depending only on $n$.

\bibliographystyle{amsplain}

\end{document}